\documentclass{amsart}

\usepackage{amscd}

\numberwithin{equation}{section}

\theoremstyle{plain} 
\newtheorem{theorem}[equation]{Theorem}
\newtheorem{lemma}[equation]{Lemma}

\newtheorem{proposition}[equation]{Proposition}
\newtheorem{corollary}[equation]{Corollary}
\newtheorem{conjecture}[equation]{Conjecture}

\theoremstyle{definition}
\newtheorem{definition}[equation]{Definition}

\newtheorem{example}[equation]{Example}
\newtheorem{remark}[equation]{Remark}
\newtheorem*{remark*}{Remark}

\newcommand{\DeltaFixedPointsTheoremText}
{For $k\geq 1$, the fixed point space $\weakfixedspecific{\Delta_{k}}{p^{k}}$ has $\TitsOrdinary{k}^\diamond$ as a retract.}

\newcommand{\GammaFixedPointsTheoremText}
{For $k\geq 1$, the fixed point space $\weakfixedspecific{\Gamma_{k}}{p^{k}}$ is homeomorphic
to $\TitsSymp{k}$.}

\newtheorem*{GammaTheorem}{Theorem~\ref{theorem: fixed points Gamma_k}}
\newtheorem*{DeltaTheorem}{Theorem~\ref{theorem: Tits building retract}}

\newtheorem*{ConjectureBranching}{Conjecture~\ref{conjecture: branching}}
\newcommand{\ConjectureBranchingText}{
There is a $U(n-1)$-equivariant homotopy equivalence
\[
\Lcal_n \simeq U(n-1)_+ \wedge_{\Sigma_n}\left(\Pcal_n^\diamond \wedge \RepresentationSphere{n}\right).
\]
}

\newtheorem*{FixedPointConjecture}{Conjecture~\ref{conjecture: fixed}}
\newcommand{\FixedPointConjectureText}{
Let $\Cwiggle=C_{\Upkfoot}\left(\Delta_{k}\right)/\left(\Delta_{k}\times S^{1}\right)$.
There is a homotopy equivalence
\begin{equation}    
\weakfixedspecific{\Delta_{k}}{p^{k}}
\simeq
\Cwiggle_{+} \wedge \TitsOrdinary{k}^{\diamond}.
\end{equation}
}

\theoremstyle{plain} 

\newcommand{\defining}[1]{{\emph{#1}}}

\DeclareMathOperator{\Unif}{Unif}

\newcommand{\Uniform}[1]{\Unif\!\weakfixedspecific{#1}{p^k}}
\newcommand{\UniformNew}[2]{\Unif\!\weakfixedspecific{#1}{#2}}

\newcommand{\isogroupof}[1]{I_{#1}}

\newcommand{\reals}{{\mathbb{R}}}
\newcommand{\integers}{{\mathbb{Z}}}

\newcommand{\complexes}{{\mathbb{C}}}
\newcommand{\field}{{\mathbb{F}}}

\newcommand{\Upk}{U\kern-2.5pt\left(p^{k}\right)}
\newcommand{\Upkfoot}{U\kern-.5pt\left(p^{k}\right)}
\newcommand{\Upp}{U\kern-.2pt(p)}
\newcommand{\Uof}[1]{U\kern-2pt\left(#1\right)}
\newcommand{\Uoffoot}[1]{U\kern-.5pt\left(#1\right)}
\newcommand{\GLof}[1]{\GL_{#1}\!\left(\field_{p}\right)}
\newcommand{\Un}{\Uof{n}}

\newcommand{\RepresentationSphere}[1]{S^{\bar{\rho}_{#1}}}

\newcommand{\weakfixed}[1]{ \left(\Lcal_{n}\right)^{#1} }
\newcommand{\weakfixedspecific}[2]{ \left(\Lcal_{#2}\right)^{#1} }

\newcommand{\glom}[2]{ #1/#2}
\newcommand{\isorefine}[2]{ #1_{\,\operatorname{iso}(#2)}}
\newcommand{\DiagGamma}[2]{\operatorname{Diag}_{#1}\kern-2.8pt\left(\Gamma_{#2}\right)}

\def\doCal#1{%
\ifx#1\doAllCalEnd\def\doAllCal{\relax}\else%
 \expandafter\edef\csname#1cal\endcsname{{\noexpand\mathcal #1}}\fi}
\def\doAllCal#1{\doCal#1\doAllCal}
\doAllCal ABCDEFGHIJHLMNOPQRSTUVWXYZ\doAllCalEnd

\def\doBar#1{%
\ifx#1\doAllBarEnd\def\doAllBar{\relax}\else%
 \expandafter\edef\csname#1bar\endcsname{{\noexpand\overline{#1}}}\fi}
\def\doAllBar#1{\doBar#1\doAllBar}
\doAllBar ABCDEFGHIJHLMNOPQRSTUVWXYZabcdefghijhlmnopqrstuvwxyz\doAllBarEnd
\def\doWiggle#1{%
\ifx#1\doAllWiggleEnd\def\doAllWiggle{\relax}\else%
 \expandafter\edef\csname#1wiggle\endcsname{{\noexpand\tilde{#1}}}\fi}
\def\doAllWiggle#1{\doWiggle#1\doAllWiggle}
\doAllWiggle ABCDEFGHIJHLMNOPQRSTUVWXYZabcdefghijklmnopqrstuvwxyz\doAllWiggleEnd

\DeclareMathOperator{\class}{cl}

\DeclareMathOperator{\Cone}{Cone}

\DeclareMathOperator{\GL}{GL}

\DeclareMathOperator{\Ind}{Ind}

\DeclareMathOperator{\Obj}{Obj}

\DeclareMathOperator{\Sp}{Sp}

\DeclareMathOperator{\Span}{Span}

\newcommand{\MacLane}{Mac\,Lane }

\newcommand{\pdash}{$p$\kern1.3pt-}

\newcommand{\TitsOrdinary}[1]{T\GL_{#1}\!\left(\field_{p}\right)}
\newcommand{\TitsSymp}[1]{T\Sp_{#1}\left(\field_{p}\right)}

\newcommand{\partitionof}[1]{\lambda_{#1}}

\DeclareMathOperator{\Sd}{Sd}
\newcommand{\twisted}[1]{\Sd_{e}\left(#1\right)}
\newcommand{\includeTits}{F}
\newcommand{\NorthCone}{{\rm{Cone}}^{+}(\Tcal)}
\newcommand{\SouthCone}{{\rm{Cone}}^{-}(\Tcal)}

\newcommand{\Lpk}{\Lcal_{p^{k}}}

\newcommand{\Cpk}{\complexes^{p^{k}}}
\newcommand{\Cn}{\complexes^{n}}
\newcommand{\tallstrut}{\rule[-.5\baselineskip]{0pt}{\baselineskip}}
\newcommand{\Nontransitive}[1]{\tallstrut\left(\Lpk\right)^{#1}_{\rm{Ntr}}}
\newcommand{\Moving}[1]{\tallstrut\left(\Lpk\right)^{#1}_{\rm{move}}}
\newcommand{\NontransitiveDisp}[1]{\left(\Lpk\right)^{#1}_{\rm{Ntr}}}
\newcommand{\MovingDisp}[1]{\left(\Lpk\right)^{#1}_{\rm{move}}}

\begin{document}

\title[Fixed points of coisotropic subgroups]{Fixed points of coisotropic subgroups of $\Gamma_{k}$ on decomposition spaces}

\author{Gregory Arone}
\address{Stockholm University}
\email{gregoryarone@gmail.com}
\author{Kathryn Lesh}
\address{Department of Mathematics, Union College, Schenectady NY}
\email{leshk@union.edu}

\date{\today}


\begin{abstract}
We study the equivariant homotopy type of the poset
$\Lpk$ of orthogonal decompositions of~$\Cpk$. The fixed point
space of the $p$-radical subgroup $\Gamma_{k}\subset\Upk$ acting on $\Lcal_{p^k}$ is shown to be homeomorphic to a symplectic Tits building, a wedge of $(k-1)$-dimensional spheres.
Our second result concerns $\Delta_{k}=(\integers/p)^{k}\subset\Upk$ acting
on $\complexes^{p^k}$ by the regular representation.
We identify a retract of the fixed point space of $\Delta_{k}$ acting on $\Lcal_{p^k}$. This retract has the homotopy type of the unreduced suspension of the Tits building for $\GLof{k}$, also a wedge of $(k-1)$-dimensional spheres.
As a consequence of these results, we find that the fixed point space of any
coisotropic subgroup of $\Gamma_{k}$ contains, as a retract, a wedge of $(k-1)$-dimensional spheres. We make a conjecture about the full homotopy
type of the fixed point space of~$\Delta_{k}$ acting on~$\Lcal_{p^k}$, based on a more general branching conjecture, and we show that the conjecture is consistent with our results.
\end{abstract}

\maketitle
\section{Introduction}
\label{section: introduction}

A \defining{proper orthogonal decomposition} of $\complexes^n$ is an unordered collection of nontrivial, pairwise orthogonal, proper vector subspaces of $\complexes^n$ whose sum is $\complexes^n$. These decompositions have a partial ordering given by coarsening and accordingly form a topological poset category, denoted~$\Lcal_{n}$.
The category $\Lcal_{n}$ has a (topological) nerve, also denoted~$\Lcal_{n}$, and we trust to context to distinguish whether by $\Lcal_{n}$ we mean the poset (a topological category) or its nerve (a simplicial space).
The action of $\Un$ on $\Cn$ induces a natural action of $\Un$ on~$\Lcal_{n}$, and we are interested in the fixed point spaces of the action of certain subgroups of~$\Un$ on $\Lcal_{n}$.

The space $\Lcal_{n}$ was introduced in~\cite{Arone-Topology}, in the context of the orthogonal calculus of M.~Weiss.
It plays an analogous role to that played
in Goodwillie's homotopy calculus by the partition complex $\Pcal_{n}$, the poset of proper nontrivial partitions of a set of $n$ elements \cite{Arone-Mahowald}.
 The space $\Lcal_{n}$ made another, related appearance in~\cite{Arone-Lesh-Crelle}, in the filtration quotients for a filtration of the spectrum $bu$ that is analogous to the symmetric power filtration of the integral Eilenberg-\MacLane spectrum. The properties of $\Lcal_{n}$ are particularly of interest in the context of the ``$bu$-Whitehead Conjecture" (\cite{Arone-Lesh-Fundamenta}  Conjecture 1.5).

The topology and some of the equivariant structure of $\Lcal_{n}$ were studied in detail in \cite{Banff1}, and~\cite{Banff2}. In particular, the goal of those papers was to determine, for a prime $p$ and for all $p$-toral subgroups $H\subseteq\Un$, whether $\weakfixed{H}$ is contractible.
This classification question is analogous to questions that had to be answered in~\cite{ADL2}, in the course of calculating the Bredon homology of~$\Pcal_{n}$. In the case of~$\Pcal_{n}$, for coefficient functors that are Mackey functors taking values in $\integers_{(p)}$-modules, the $p$-subgroups of $\Sigma_{n}$ with non-contractible fixed point spaces on $\Pcal_{n}$ present obstructions to $\Pcal_{n}$ having the same Bredon homology as a point. Fixed point spaces of subgroups of $\Sigma_{n}$ acting on $\Pcal_{n}$ were further studied in~\cite{Arone-Branching} and~\cite{Arone-Brantner}.
In a different context, the spaces $\Lcal_{n}$ were used in \cite{Hopkins-Lawson} to develop an obstruction theory for the existence of multiplicative complex orientations.

As was the case for~$\Pcal_{n}$, one expects that \pdash toral subgroups of $\Un$ acting on $\Lcal_{n}$ with non-contractible fixed point spaces will present obstructions to $\Lcal_{n}$ having the same Bredon homology as a point, at least
for coefficients that are Mackey functors taking values in $\integers_{(p)}$-modules.
In this paper, we contribute to the understanding of these fixed point spaces by identifying two key cases of \pdash toral subgroups of $\Upk$ whose fixed point spaces on $\Lcal_{p^k}$ are not only non-contractible, but actually have homology that is either free abelian or has a free abelian summand. When we put these together with a join formula from \cite{Banff2}, we obtain a similar result for all coisotropic subgroups of~$\Gamma_{k}$.

Our results have a similar flavor to results of \cite{Arone-Dwyer} and \cite{ADL2} in
that they involve Tits buildings.
 We also show that the results obtained are consistent with a more general conjecture about the equivariant homotopy type of~$\Lcal_{n}$. The conjecture is analogous to the branching rule
of~\cite{Arone-Branching} and~\cite{Arone-Brantner} for~$\Pcal_{n}$.

The results of the current work are used in \cite{Banff2} to give a complete classification of \pdash toral subgroups of $\Un$ with contractible fixed point spaces on~$\Lcal_{n}$.
Unlike the case for~$\Pcal_{n}$, where many elementary abelian \pdash subgroups of $\Sigma_{n}$ have non-contractible fixed point spaces (\cite{Arone-Branching}, \cite{Arone-Brantner}), it turns out that the fixed point spaces of the actions of most \pdash toral subgroups of $\Un$ on~$\Lcal_{n}$ are actually contractible. \cite{Banff2} shows that the only possible exceptions occur when $n=q^{i}p^{j}$,
where $q$ is a prime different from~$p$.
Theorems~\ref{theorem: fixed points Gamma_k} and~\ref{theorem: Tits building retract}
below are used in \cite{Banff2} to settle these cases.

To state our results explicitly, we need some notation for the two \pdash toral subgroups that we study. First, let $\Delta_{k}$ denote the subgroup $\left(\integers/p\right)^{k}\subset\Upk$ where
$\left(\integers/p\right)^{k}$ acts on $\Cpk$ by the regular representation. Associated to
$\Delta_{k}$ is the Tits building for $\operatorname{GL}_k(\field_p)$,
denoted~$\TitsOrdinary{k}$, which is the poset
of proper, nontrivial subgroups of~$\Delta_{k}$, and which has the homotopy type of a
wedge of spheres.

Second, let $\Gamma_{k}$ be the irreducible projective elementary abelian \pdash subgroup of $\Upk$ (unique up to conjugacy), which is given by an extension
\begin{equation}  \label{eq: Gamma_k}
1 \rightarrow S^{1} \rightarrow \Gamma_{k}\rightarrow(\integers/p)^{2k}\rightarrow 1.
\end{equation}
Here $S^{1}$ denotes the center of~$\Upk$.
(See Section~\ref{section: background} for a brief discussion of~$\Gamma_{k}$, or \cite{Oliver-p-stubborn} or \cite{Banff2}
for a detailed discussion from basic principles.)
The extension \eqref{eq: Gamma_k} induces a symplectic form on~$(\integers/p)^{2k}$ defined
by lifting to $\Gamma_{k}$ and taking the commutator, which lies in $S^{1}$ and has order~$p$.
We define a coisotropic subgroup of~$\Gamma_{k}$ to be a subgroup
that is the preimage in \eqref{eq: Gamma_k} of a coisotropic subspace of~$(\integers/p)^{2k}$. (See Definition~\ref{definition: coisotropic}.)
This allows us to associate to $\Gamma_{k}$ the Tits building for the symplectic group,
 denoted~$\TitsSymp{k}$, which is the poset of proper coisotropic subgroups of~$(\integers/p)^{2k}$, and like $\TitsOrdinary{k}$ has the homotopy type of a wedge of spheres.

We have two main results.

\begin{theorem}  \label{theorem: fixed points Gamma_k}
\GammaFixedPointsTheoremText
\end{theorem}

By way of context, we point out that there is a more elementary, analogous result
to Theorem~\ref{theorem: fixed points Gamma_k} that establishes
a homeomorphism between the fixed point space of the action of
 $\Delta_k\subset\Sigma_{p^k}$ on the partition complex~$\Pcal_{p^k}$, and the Tits building for $\GL_k(\field_p)$
 (\cite[Lemma 10.1]{ADL2}).
The paper~\cite{Arone-Topology} establishes a dictionary between the properties of the action of $\Sigma_{n}$ on $\Pcal_{n}$ and the action of $\Un$ on $\Lcal_n$. The dictionary translates $\Delta_{k}\subset\Sigma_{p^k}$ to $\Gamma_{k}\subset\Upk$,
and translates the Weyl group of $\Delta_{k}$ in $\Sigma_{k}$, which is $\GL_k(\field_p)$, to the Weyl group of $\Gamma_{k}$ in $\Upk$, which is
$\Sp_{k}(\field_p)$. Therefore
Theorem~\ref{theorem: fixed points Gamma_k} is the result one would expect to get by taking the dictionary literally.

On the other hand, the dictionary does not give a prediction
for~$\weakfixedspecific{\Delta_{k}}{p^k}$, although we explain later how the following theorem is consistent with a more general conjecture.
Given a space~$X$, let $X^{\diamond}$ denote the unreduced suspension of~$X$.

\begin{theorem}     \label{theorem: Tits building retract}
\DeltaFixedPointsTheoremText
\end{theorem}

We compute explicit examples for $k=1$ in
Examples~\ref{example: symplectic}
and~\ref{example: general linear}.

\begin{remark}
As part of proving Theorem~\ref{theorem: Tits building retract},
 we need to construct an inclusion $\TitsOrdinary{k}^\diamond \hookrightarrow \weakfixedspecific{\Delta_{k}}{p^{k}}$, as well as a retraction. Constructing the inclusion is perhaps the sneakiest step in the paper.
 Contrary to what one might expect, the inclusion is not induced by a functor from the poset of subspaces of $\field_p^k$ to the poset of decompositions of $\complexes^{p^k}$. (Note that in any case, we want the suspension of the Tits building as the retract, and not the Tits building itself.)
 Rather, we need to use the \defining{edgewise subdivision} of the poset of subspaces of $\field_p^k$ to model the space $\TitsOrdinary{k}^\diamond$.
 The edgewise subdivision is a poset whose objects are nested pairs $(H\subseteq K)$ of subgroups of~$\Delta_k$, and whose morphisms are twisted arrows. We construct a functor from the edgewise subdivision to the poset of
 decompositions of $\complexes^{p^k}$ using a mixture of the action of $K$ on a basis of $\complexes^{p^k}$, and canonical decomposition of $H$-representations into $H$-isotypical summands.
 Details appear in the latter part of Section~\ref{section: Delta_k fixed points}.
\end{remark}

With Theorems~\ref{theorem: fixed points Gamma_k}
and~\ref{theorem: Tits building retract} in place, we can use a join formula from \cite{Banff2} to identify a wedge of spheres as a retract of the fixed point space of any coisotropic subgroup of~$\Gamma_{k}$
acting on~$\Lcal_{p^{k}}$.

\begin{corollary}    \label{corollary: retract spheres}
If $k\geq 1$ and $H\subseteq\Gamma_{k}$ is coisotropic, then $\weakfixedspecific{H}{p^{k}}$ has a retract that is homotopy equivalent to a wedge of spheres of dimension~$k-1$.
\end{corollary}

\begin{proof}
Because $H$ is coisotropic, it has the form $\Gamma_{s}\times\Delta_{t}$ for some $s+t=k$
(Lemma~\ref{lemma: form of coisotropic subgroups}). If $s=k$ or $t=k$,
then the result is the same as Theorems~\ref{theorem: fixed points Gamma_k} and~\ref{theorem: Tits building retract}, respectively, because
$\TitsSymp{k}$ and $\TitsOrdinary{k}^\diamond$ are both wedges of
$(k-1)$-dimensional spheres. (See, for example,
\cite[Theorem 4.127]{Abramenko-Brown-Buildings}.)
By direct computation, this statement includes the case
$t=k=1$,
since $\TitsOrdinary{1}$ is empty and so its unreduced suspension is a $0$-sphere, as required.

When $s$ and $t$ are both less than $k$, we apply
\cite{Banff2} Theorem~9.2 to find that
\[
\weakfixedspecific{H}{p^{k}}\cong\weakfixedspecific{\Delta_{t}}{p^{t}}\ast
    \weakfixedspecific{\Gamma_{s}}{p^{s}}.
\]
Hence by Theorems~\ref{theorem: fixed points Gamma_k}
and~\ref{theorem: Tits building retract},
$\weakfixedspecific{H}{p^{k}}$ has
$\TitsOrdinary{t}^\diamond\ast\TitsSymp{s}$ as a retract.
But $s$ and $t$ must both be at least~$1$, and
as noted above, the spaces $\TitsOrdinary{t}^{\diamond}$ and $\TitsSymp{s}$ each have the homotopy type of a wedge of spheres,
of dimension~$t-1$ and~$s-1$, respectively.
The corollary follows because $s+t=k$.
\end{proof}

Theorem~\ref{theorem: Tits building retract} is good  enough to complete the classification of~\cite{Banff2}: all that is needed there is that the integral homology of
$\weakfixedspecific{\Delta_{k}}{p^{k}}$ has a summand that is a free abelian group. However, we actually have a conjectural description of the full homotopy type of the fixed point space~$\weakfixedspecific{\Delta_{k}}{p^{k}}$, based on a more general conjecture
regarding the equivariant homotopy type of~$\Lcal_{n}$.
We can embed $\Uof{n-1}\subseteq\Un$ (in a nonstandard way) as the symmetries of
the orthogonal complement of the diagonal $\complexes\subset\complexes^{n}$, since that
complement is an $(n-1)$-dimensional vector space over~$\complexes$.
Observe that the standard inclusion
$\Sigma_{n}\hookrightarrow\Un$ by permutation matrices actually factors
through this inclusion $\Uof{n-1}\subset\Un$. Finally, let $\RepresentationSphere{n}$ denote the one-point compactification of the reduced standard representation of $\Sigma_{n}$ on~$\reals^{n-1}$.
The general conjecture is as follows.

\begin{conjecture}\label{conjecture: branching}
\ConjectureBranchingText
\end{conjecture}

\begin{remark}
Conjecture~\ref{conjecture: branching} is motivated by the role of $\Lcal_n$ in orthogonal calculus. On the one hand, $\Lcal_n$ is closely related to the $n$-th derivative of the functor $V\mapsto BU(V)$. This, together with the fibration sequence $S^1\wedge S^V\to BU(V)\to BU(V\oplus \complexes)$ implies that the restriction of $\Lcal_n$ to $U(n-1)$ is closely related to the $n$-th derivative of the functor $V\mapsto S^1\wedge S^V$.
On the other hand, by connection with Goodwillie's homotopy calculus, the $n$-th derivative of this last functor is closely related to $ \Pcal_n^\diamond \wedge \RepresentationSphere{n}$. In fact, one can use this connection to prove that the equivalence in Conjecture~\ref{conjecture: branching} is true after taking
suspension spectra and smashing with~$EU(n)_+$. For more details see~\cite{Arone-Topology}, especially Theorem 3, which is equivalent to the assertion of the previous sentence,  modulo standard manipulations involving Spanier-Whitehead duality.
\end{remark}

In the final section of this paper, we show what the general statement in Conjecture~\ref{conjecture: branching} would imply about the actual homotopy type of
$\weakfixedspecific{\Delta_{k}}{p^{k}}$, and we check that implication against what we can prove beginning from
Theorems~\ref{theorem: fixed points Gamma_k}
and~\ref{theorem: Tits building retract}.
Let $C_{G}(H)$ denote the centralizer of a subgroup $H$ in a group~$G$.
After some calculation, we find that
Conjecture~\ref{conjecture: branching} implies the following conjecture.

\begin{conjecture}\label{conjecture: fixed}
\FixedPointConjectureText
\end{conjecture}

We observe that Conjecture~\ref{conjecture: fixed} is consistent with
Theorem~\ref{theorem: Tits building retract}, and this consistency
can be regarded as evidence for the correctness of Conjecture~\ref{conjecture: branching}.

\bigskip

\noindent
\emph{Organization of the paper}\\
\indent In Section~\ref{section: background}, we collect some background information about~$\Lcal_{n}$, the $p$-toral group~$\Gamma_{k}$, and the symplectic Tits building.
Section~\ref{section: Gamma_k fixed points} proves
Theorem~\ref{theorem: fixed points Gamma_k} and computes
an example.
Section~\ref{section: Delta_k fixed points} proves
Theorem~\ref{theorem: Tits building retract}, and lastly, in
Section~\ref{section: conjectures} we show how to deduce
Conjecture~\ref{conjecture: fixed} from Conjecture~\ref{conjecture: branching},
and we compute another example.

Throughout the paper, we assume that we have fixed a prime~$p$. By a subgroup of a Lie group, we always mean a closed subgroup. We write $N_G(H)$ and $C_G(H)$ for the normalizer and centralizer, respectively, of a subgroup $H$ in a group~$G$. The notation $S^{1}$ always means the center of the unitary group under discussion. We write $\rho_{n}$ for the standard representation of $\Sigma_{n}$ on~$\complexes^{n}$, and we write $\bar\rho_{n}$ for the reduced standard representation (the quotient of the standard representation by the trivial representation).

\section{Background on $\Lcal_{p^{k}}$ and $\Gamma_{k}$}
\label{section: background}

In this section, we give background results on the decomposition spaces~$\Lcal_{n}$, the group~$\Gamma_{k}$, and the symplectic Tits building.

As explained in Section~\ref{section: introduction}, $\Lcal_{n}$ is
a poset category internal to topological spaces: the objects and morphisms have an action of~$\Un$ and are topologized as disjoint unions of $\Un$-orbits.
If $\lambda$ is an object of~$\Lcal_{n}$, then we write $\class(\lambda)$ for the set of subspaces that make up~$\lambda$, which are called the \defining{classes} or \defining{components}
of~$\lambda$. If a decomposition $\lambda$ is stabilized by the action of a subgroup $H\subseteq\Un$, then there is an action of $H$ on $\class(\lambda)$, which may be nontrivial.

In analyzing $\weakfixed{H}$, there are two operations that are particularly helpful in constructing deformation retractions to subcategories.

\begin{definition}
\label{definition: operations on decompositions}
Suppose that $H\subseteq\Un$ is a closed subgroup, and $\lambda$ is a decomposition in $\weakfixedspecific{H}{n}$.
\mbox{}\hfill
\begin{enumerate}
\item We define $\glom{\lambda}{H}$ as the decomposition of $\complexes^{n}$ obtained by summing components of $\class(\lambda)$ that are in the same orbit of the action of $H$ on~$\class(\lambda)$.
\item If $\mu$ is a decomposition of $\Cn$ such that $H$ acts trivially on~$\class(\mu)$ (i.e., every component of $\mu$ is a representation of~$H$), then we define $\isorefine{\mu}{H}$ as the refinement of $\mu$ obtained by taking the canonical decomposition of each component of $\mu$ into its $H$-isotypical summands.
\end{enumerate}
\end{definition}

\begin{example}
Let $\{e_{1}, e_{2}, e_{3}, e_{4}\}$ denote the standard basis for~$\complexes^{4}$, and let $\Sigma_{4}\subset\Uof{4}$ act by permuting the basis vectors. Let $\epsilon$ denote the decomposition of $\complexes^{4}$ into the four lines determined by the standard basis. Let $H\cong\integers/2\subset\Sigma_{4}$ be generated by $(1,2)(3,4)$. Then $\mu:=\glom{\epsilon}{H}$ consists of two components $v_{1}=\Span\{e_{1},e_{2}\}$ and
$v_{2}=\Span\{e_{3},e_{4}\}$.

Since each component of $\mu$ is a representation of~$H$, we can refine $\mu$ as $\isorefine{\left(\glom{\epsilon}{H}\right)}{H}$. Each of the components $v_{1}$ and $v_{2}$ decompose into one-dimensional eigenspaces for the action of~$H$, one for the eigenvalue $+1$ and one for the eigenvalue $-1$. Hence
$\isorefine{\left(\glom{\epsilon}{H}\right)}{H}$ is a decomposition of $\complexes^{4}$ into four lines, each of which is fixed by~$H$, where $H$ acts on two of them by the identity and on the other two by multiplication by~$-1$.
\end{example}

Since $\Lcal_{n}$ has a topology, it is necessary that the operations of
Definition~\ref{definition: operations on decompositions} be continuous, which is proved in \cite{Banff2} using the following lemma, specialized from \cite[Lemma 1.1]{May-Tel-Aviv}.

\begin{lemma}     \label{lemma: path components}
The path components of the object and morphism spaces of $\weakfixed{H}$ are orbits of the identity component of the centralizer of $H$ in~$\Un$.
\end{lemma}

The proof of continuity of the operations of
Definition~\ref{definition: operations on decompositions} then goes by observing that the operations commute with the action of the centralizer of~$H$ in~$\Un$, which defines the topology of~$\weakfixed{H}$, since the orbits of $\Un$ determine the topology of~$\Lcal_{n}$. See \cite{Banff2} Section~4.

Our next job is to identify a smaller subcomplex of $\weakfixed{H}$ that
is sometimes good enough to compute the homotopy type of~$\weakfixed{H}$.

\begin{definition}
Let $H\subseteq\Un$ be a subgroup and suppose that $\lambda$ is a decomposition in~$\weakfixed{H}$.
\begin{enumerate}
\item For $v\in\class(\lambda)$, we define the \defining{$H$-isotropy group of $v$}, denoted $\isogroupof{v}$, as $\isogroupof{v}=\{h\in H: hv=v\}$.
\item We say that $\lambda$ has \defining{uniform $H$-isotropy} if all elements of $\class(\lambda)$ have the same $H$-isotropy group. In this case, we write $\isogroupof{\lambda}$ for the $H$-isotropy group of any $v\in\class(\lambda)$, provided that the group $H$ is understood from context.
\end{enumerate}
\end{definition}

\begin{example}    \label{example: Gamma_k isotropy uniform}
Suppose that $\lambda\in\Obj\weakfixed{H}$, and that $H$ acts transitively on the set~$\class(\lambda)$. If there exists $v\in\class(\lambda)$ such that $\isogroupof{v}\triangleleft H$, then $\lambda$ necessarily has uniform $H$-isotropy. This is because the transitive action of $H$ means that the $H$-isotropy groups of
all components of $\lambda$ are conjugate in~$H$. Since $\isogroupof{v}$ is normal, all the isotropy groups are actually the same.

More specifically, suppose that $H\subset\Un$ has the property that $H/(H\cap S^{1})$ is abelian (resp., elementary abelian), where $S^1$ denotes the center of~$\Un$. In this case we say that $H$ is \defining{projective abelian} (resp., \defining{projective elementary abelian}).  By the discussion above, if
$\lambda\in\Obj\weakfixed{H}$ has a transitive action of $H$ on $\class(\lambda)$, then $\lambda$ has uniform $H$-isotropy, because every subgroup of $H$ containing $H\cap S^{1}$ is normal.
\end{example}

For $H\subset\Un$, let $\UniformNew{H}{n}$ denote the subposet of $\weakfixed{H}$ consisting of objects with uniform $H$-isotropy. As in~\cite{Banff2}, we have the following lemma, stated slightly more generally here.

\begin{lemma}  \label{lemma: inclusion of uniform}
If $H\subset\Un$ is a projective abelian subgroup, then the inclusion $\UniformNew{H}{n}\rightarrow\weakfixed{H}$ induces a homotopy equivalence of nerves.
\end{lemma}

\begin{proof}
Exactly the same proof as in \cite{Banff2} works here. Let $\lambda$ be a decomposition in~$\weakfixed{H}$, with  $\class(\lambda)=\{v_{1},...,v_{j}\}$. Each $\isogroupof{v_i}$ contains $H\cap S^{1}$, and so is normal in~$H$ because $H$ is projective abelian. Let $J_{\lambda}=\isogroupof{v_1}...\isogroupof{v_j}$,
which is also a normal subgroup of~$H$. We assert that
$\glom{\lambda}{J_{\lambda}}$ is a proper decomposition. If not,
then $J_{\lambda}$ (and hence also~$H$) acts transitively on $\class(\lambda)$. A transitive action of~$H$ on $\class(\lambda)$
would tell us that  $\isogroupof{v_1}=...=\isogroupof{v_j}=J_{\lambda}$, and that
$I_{v_{1}}$, for example, acts transitively on $\class(\lambda)$. However, $I_{v_{1}}$ fixes $v_{1}$, so $\lambda$ would have
only have one component, a contradiction.

From this point, the proof is precisely as in \cite{Banff2}, by doing the routine checks that $\lambda\mapsto\glom{\lambda}{J_{\lambda}}$ is a continuous deformation retraction from $\weakfixed{H}$ to $\UniformNew{H}{n}$.
\end{proof}

Our next order of business is to provide a little background on the groups whose fixed points we study in this paper. As in the introduction, we write $\Delta_{k}$ for the group $(\integers/p)^{k}\subset\Upk$ acting on the standard basis of $\Cpk$ by the regular representation.
One of the goals of this paper is to understand the fixed point space of $\Delta_{k}$ acting on $\Lcal_{p^k}$
(Theorem~\ref{theorem: Tits building retract} and Conjecture~\ref{conjecture: fixed}).

The other important group in our results is $\Gamma_{k}\subset\Upk$, which denotes a subgroup of $\Upk$ given by an extension
\[
1\rightarrow S^{1}\rightarrow\Gamma_{k}\rightarrow
         \left(\integers/p\right)^{k}\times\left(\integers/p\right)^{k}\rightarrow 1,
\]
and, of key importance, acts irreducibly on~$\complexes^{p^k}$.
The group $\Gamma_{k}$ is discussed extensively and described explicitly in terms of matrices in \cite{Oliver-p-stubborn}. (See also \cite{Banff2} for a discussion from first principles.) Each factor of $\left(\integers/p\right)^{k}$ has a splitting back into $\Gamma_{k}$, though the splittings of the two factors do not commute in~$\Gamma_{k}$.
As a subgroup of $\Gamma_{k}\subseteq\Upk$, the image of the splitting of the first factor
of $\left(\integers/p\right)^{k}$ can be regarded as $\Delta_{k}$ itself, acting on the standard basis of $\Cpk$ by the regular representation. The image of the splitting of the second factor of $\left(\integers/p\right)^{k}$ acts via the regular representation on the
$p^k$ one-dimensional irreducible representations of~$\Delta_{k}$, which are pairwise nonisomorphic and span~$\Cpk$.

Moving on to Tits buildings, recall that a symplectic form on an $\field_{p}$-vector space is a nondegenerate alternating bilinear form. The vector space necessarily has even dimension.
Lifting elements of $\Gamma_{k}/S^{1}$ to $\Gamma_{k}$ and computing the commutator gives a well-defined symplectic form on $\left(\integers/p\right)^{k}\times\left(\integers/p\right)^{k}$. Oliver shows in \cite{Oliver-p-stubborn} that the Weyl group of $\Gamma_{k}$
in~$\Upk$ is the full group of automorphisms of this form, that is,
the Weyl group of $\Gamma_{k}$ in $\Upk$ is the symplectic group $\Sp_{k}\left(\field_{p}\right)$.
Our next goal is to describe the symplectic Tits
building,~$\TitsSymp{k}$.

\begin{definition}   \label{definition: coisotropic}
\mbox{}\hfill
\begin{enumerate}
\item A subspace $W$ of a symplectic vector space is called \defining{coisotropic} if $W^{\perp}\subseteq W$.
\item We say that $J\subseteq\Gamma_{k}$ is a \defining{coisotropic subgroup} if $J$ is the inverse image of a coisotropic subspace of~$(\integers/p)^{2k}$.
\item
The \defining{symplectic Tits building}, $\TitsSymp{k}$, is the poset of proper coisotropic subgroups of~$\Gamma_{k}$.
\end{enumerate}
\end{definition}

\begin{example} \label{example: Tits building example}
To compute $\TitsSymp{1}$, consider
\[
1\rightarrow S^{1}\rightarrow\Gamma_{1}
        \rightarrow\left(\integers/p\right)^{2}
        \rightarrow 1.
\]
Coisotropic subspaces have dimension at least half the dimension of the ambient vector space, so here a proper coisotropic subspace of $\left(\integers/p\right)^{2}$ has dimension one. Further, every one-dimensional subspace of a two-dimensional symplectic vector space is coisotropic. The vector space $(\integers/p)^2$ has $p+1$ one-dimensional subspaces.
Since there are no possible inclusions between the subspaces, there are no morphisms in the poset, and therefore the nerve of $\TitsSymp{1}$ consists of $p+1$ isolated points.
\end{example}

\begin{remark*}
In the literature, the symplectic Tits building is usually defined in terms of {\it isotropic} subspaces.
The poset of flags of isotropic subspaces is isomorphic to the poset of parabolic subgroups of the symplectic group $\Sp_{k}(\field_p)$, and this is why its geometric realization
is identified with the symplectic Tits building.
In general, $\TitsSymp{k}$ has the homotopy type of a wedge of spheres of dimension~$k-1$.
 See~\cite[Section 6.6]{Abramenko-Brown-Buildings} for more details.
Taking orthogonal complement defines a canonical (inclusion-reversing) bijection between isotropic and coisotropic subspaces, and
for our purposes it is more natural to focus on the coisotropic subgroups.
\end{remark*}

Our final piece of background is some concrete information about coisotropic subgroups. Let
${\mathbb{H}}_{s}$ denote an $2s$-dimensional vector space over $\integers/p$ with a symplectic form, and let ${\mathbb{T}}_{t}$ denote a $t$-dimensional vector space with trivial form.

\begin{lemma}  \label{lemma: form of coisotropic subgroups}
If $H\subseteq\Gamma_{k}$ is coisotropic, then $H$ has the form $\Gamma_{s}\times\Delta_{t}$ where $s+t=k$.
\end{lemma}

\begin{proof}
A coisotropic subspace of $(\integers/p)^{2k}$ has an alternating form isomorphic to ${\mathbb{H}}_{s}\oplus{\mathbb{T}}_{t}$ where $s+t=k$.
Further, $H$ is classified up to isomorphism by its commutator form, with
${\mathbb{H}}_{s}$ corresponding to $\Gamma_{s}$ and ${\mathbb{T}}_{t}$ corresponding
to~$\Delta_{t}$. (A proof is given in \cite{Banff2}.)
The result follows.
\end{proof}

\begin{lemma}   \label{lemma: coisotropic determined by size}
If $H\subseteq\Gamma_{k}$ is coisotropic, then $H$ has irreducibles of dimension $p^s$ $\iff$ $H\cong\Gamma_{s}\times\Delta_{t}$ where $s+t=k$.
\end{lemma}

\begin{proof}
We already know from Lemma~\ref{lemma: form of coisotropic subgroups}
that $H$ is isomorphic to $H\cong\Gamma_{s}\times\Delta_{t}$ where $s+t=k$. The lemma follows from the fact that $\Gamma_{s}$ is acting on $\Cpk$ by a multiple of the
standard representation, and the irreducible representations of
$\Gamma_{s}\times\Delta_{t}$ are products of
irreducible representations of $\Gamma_{s}$ and (one-dimensional) irreducible
representations of~$\Delta_{t}$.
\end{proof}

\section{Fixed points of $\Gamma_{k}$ acting on $\Lcal_{p^k}$}
\label{section: Gamma_k fixed points}

In this section, we prove the first theorem announced in the introduction.

\begin{GammaTheorem}
\GammaFixedPointsTheoremText
\end{GammaTheorem}

To get a feel for the result, we begin by computing the case $k=1$ of
Theorem~\ref{theorem: fixed points Gamma_k} directly.

\begin{example}     \label{example: symplectic}
To compute $\weakfixedspecific{\Gamma_{1}}{p}$, suppose that
$\lambda$ is a decomposition of $\complexes^{p}$ that is fixed
by~$\Gamma_{1}$. Because $\Gamma_{1}$ acts irreducibly on~$\complexes^{p}$, the action of $\Gamma_{1}$ on $\class(\lambda)$
is transitive, meaning that $\class(\lambda)$ has one element, $p^2$ elements, or $p$~elements. The first is impossible because
$\lambda$ is proper (must have more than one class), and the second
is impossible because classes of $\lambda$ must be nonzero (cannot have $p^{2}$ nonzero classes in a decomposition of $\complexes^{p}$). Hence $\lambda$ is a decomposition of $\complexes^{p}$ into $p$ lines. The kernel $\isogroupof{\lambda}$ of the action map $\Gamma_{1}\rightarrow\Sigma_{\class(\lambda)}$ has the form
$\isogroupof{\lambda}\cong S^{1}\times\integers/p$. The decomposition $\lambda$ is exactly the canonical decomposition of $\complexes^{p}$ into $\isogroupof{\lambda}$-isotypical representations. Hence there is a one-to-one correspondence between subgroups $\isogroupof{}\cong S^{1}\times\integers/p$ of $\Gamma_{1}$ and $\Gamma_{1}$-invariant decompositions $\lambda$
of~$\complexes^{p}$. There are $p+1$ subgroups $\isogroupof{}$ of the required form, and there are no possible inclusions, so $\weakfixedspecific{\Gamma_{1}}{p}$ consists of $p+1$ points.
Comparing to Example~\ref{example: Tits building example}, we
see that $\TitsSymp{1}$ also consists of $p+1$ isolated points,
as required by Theorem~\ref{theorem: fixed points Gamma_k}.
\end{example}

\smallskip
Example~\ref{example: symplectic} brings up the point that while $\TitsSymp{k}$ is a discrete poset, it is not initially clear that $\weakfixedspecific{\Gamma_{k}}{p^k}$ is discrete, because $\Lcal_{p^{k}}$ itself is a topological poset. While it is not logically necessary to verify discreteness up front, we give a freestanding proof that $\weakfixedspecific{\Gamma_{k}}{p^k}$ is a discrete poset.

\begin{lemma}   \label{lemma: fixed points discrete}
The object and morphism spaces of $\weakfixedspecific{\Gamma_{k}}{p^k}$ are discrete.
\end{lemma}

\begin{proof}
By Lemma~\ref{lemma: path components}, the path components of
$\Obj\left(\Lcal_{p^k}\right)^{\Gamma_{k}}$ are orbits of the centralizer of $\Gamma_{k}$ in~$\Upk$. However, $\Gamma_{k}$ is centralized in $\Upk$ only by the center~$S^{1}$ of~$\Upk$ \cite[Prop.~4]{Oliver-p-stubborn}). Since $S^{1}$ actually fixes every object of $\Lcal_{p^{k}}$, the $S^{1}$-orbit of an object of $\Lcal_{p^k}$ is
just a point. Hence the path components of the object space of $\weakfixedspecific{\Gamma_{k}}{p^{k}}$  are single points,
and the object space of $\weakfixedspecific{\Gamma_{k}}{p^{k}}$ is discrete. The same is
then necessarily true of the morphism space, since there is at most one morphism
between any two objects and the source and target maps are continuous on the morphism space.
\end{proof}

The strategy for the proof of
Theorem~\ref{theorem: fixed points Gamma_k}
is straightforward: to establish functors from $\TitsSymp{k}$ to $\weakfixedspecific{\Gamma_{k}}{p^k}$ and back, and to show that their compositions are identity functors.
Defining the functions on objects is not difficult. To show that the maps are functorial and compose to identity functors requires some representation theory.

We will define functions in both directions between the proper coisotropic subgroups of~$\Gamma_{k}$ and the objects of~$\weakfixedspecific{\Gamma_{k}}{p^k}$. If $H$ is a subgroup of~$\Gamma_{k}$, let $\partitionof{H}$ denote the canonical decomposition of $\complexes^{p^k}$ by $H$-isotypical summands. On the other hand, recall that if $\lambda$ is an object of $\weakfixedspecific{\Gamma_{k}}{p^k}$, then $\lambda$ necessarily has uniform $\Gamma_{k}$-isotropy (Example~\ref{example: Gamma_k isotropy uniform}, because $\Gamma_{k}$ acts irreducibly on~$\Cpk$). We denote this isotropy by~$\isogroupof{\lambda}\subset\Upk$.
We define the required correspondences between subgroups and decompositions as follows: if $H$ is a coisotropic subgroup of~$\Gamma_{k}$, then
\begin{align*}
F(H)&=\partitionof{H}\\
\intertext{and if $\lambda$ is a decomposition in $\weakfixedspecific{\Gamma_{k}}{p^k}$, then}
G(\lambda)&=\isogroupof{\lambda}.
\end{align*}
We need to check that the image of $F$ consists of proper decompositions of~$\Cpk$, that the image of $G$ consists of proper coisotropic subgroups, that $F$ and $G$ are functorial, and that $F$ and $G$ are inverses of each other when $F$ is restricted to proper coisotropic groups.

To show that $F$ and $G$ are functors, we need a representation-theoretic lemma.

\begin{lemma}     \label{lemma: coisotropic repns not iso}
If $H$ is a coisotropic subgroup of~$\Gamma_{k}$, then the standard representation of $\Gamma_{k}$ on $\complexes^{p^k}$ breaks into the sum of $\left[\Gamma_{k}:H\right]$
irreducible representations of~$H$, all of equal dimension, and pairwise non-isomorphic.
\end{lemma}

\begin{proof}
Direct computation of the character of $\Gamma_{k}$ from the matrix representation in \cite{Oliver-p-stubborn} establishes that
$\chi_{\strut\Gamma_{k}}(x)=0$ for $x\notin S^{1}$ and
$\chi_{\strut\Gamma_{k}}(s)=p^{k}s$ for $s\in S^{1}$, and hence the same is true for the character of~$H$.
By Lemma~\ref{lemma: form of coisotropic subgroups}, we know $H\cong\Gamma_{s}\times\Delta_{t}$ with $s+t=k$. Computing the characters shows that the action of $H\cong\Gamma_{s}\times\Delta_{t}$ on $\complexes^{p^k}\cong\complexes^{p^{s}}\otimes\complexes^{p^{t}}$
is conjugate to the action where $\Gamma_{s}$ acts on the first factor by the standard representation and $\Delta_{t}$ acts on the second factor by the regular representation.
Since $H$ is a product, irreducible $H$-representations are obtained as tensor products of irreducible representations of $\Gamma_{s}$ and of~$\Delta_{t}$. There are $p^t=\left[\Gamma_{k}:H\right]$ irreducibles of $\Delta_{t}$ acting on~$\complexes^{p^{t}}$, all non-isomorphic, and the tensor products of these irreducibles with the standard representation of $\Gamma_{s}$ are again irreducible,
span~$\complexes^{p^k}$, and are pairwise non-isomorphic (for example, since they have different characters).
\end{proof}

We obtain the following corollary to
Lemma~\ref{lemma: coisotropic repns not iso}.

\begin{corollary}    \label{corollary: unique isotypical}
If $J\subseteq\Gamma_{k}$ is coisotropic, then $\partitionof{J}$ is the only $J$-isotypical decomposition of~$\Cpk$.
\end{corollary}

\begin{proof}
A decomposition of $\complexes^{p^k}$ is $J$-isotypical if and only if each one of its components is an isotypical representation of $J$. Every $J$-isotypical decomposition of $\complexes^{p^k}$ is a refinement of $\partitionof{J}$. But by Lemma~\ref{lemma: coisotropic repns not iso}, each component of $\partitionof{J}$ is irreducible. Hence $\partitionof{J}$ has no $J$-isotypical refinements, and therefore it is the only $J$-isotypical decomposition of $\complexes^{p^k}$.
\end{proof}

With Corollary~\ref{corollary: unique isotypical} in hand, we can establish that $F$ is functorial.

\begin{proposition}
$F$ is a functor from $\TitsSymp{k}$ to $\weakfixedspecific{\Gamma_{k}}{p^k}$.
\end{proposition}

\begin{proof}
Suppose $H$ is an object of $\TitsSymp{k}$, that is, a proper coisotropic subgroup of~$\Gamma_{k}$. Since $H\triangleleft\Gamma_{k}$, the action of $\Gamma_{k}$ on $\Cpk$ permutes the irreducible representations of $H$ and hence stabilizes $\partitionof{H}$ (while possibly permuting its components). Further, by Lemma~\ref{lemma: coisotropic repns not iso}, $\partitionof{H}$ has $\left[\Gamma_{k}:H\right]>1$ components, so $\partitionof{H}$ is a proper decomposition of~$\Cpk$.

To establish naturality, suppose that $J\subseteq H$ are two coisotropic subgroups of~$\Gamma_{k}$. Every component of $\partitionof{H}$ is a representation of~$H$, and hence also of~$J$.
Consider the decomposition~$\isorefine{\left(\partitionof{H}\right)}{J}$. It is $J$-isotypical, by definition, and so by Corollary~\ref{corollary: unique isotypical}, we know that $\isorefine{\left(\partitionof{H}\right)}{J}=\partitionof{J}$. It follows that $\partitionof{J}$ is a refinement of $\partitionof{H}$, so $F$ is a functor
on the poset of proper coisotropic subgroups of~$\Gamma_{k}$.
\end{proof}

Next we turn our attention to the function $G$ from objects of $\weakfixedspecific{\Gamma_{k}}{p^k}$ to subgroups of~$\Gamma_{k}$.
By way of preparation, we need a key representation-theoretic result similar to Lemma~\ref{lemma: coisotropic repns not iso}. Given an irreducible representation $\sigma$ of a group $G$ and
another representation $\tau$ of~$G$, let $[\tau:\sigma]$ denote
the multiplicity of $\sigma$ in~$\tau$.

\begin{lemma} \label{lemma: non-iso restrictions moved}
   Let $\lambda$ be an object of $\weakfixedspecific{\Gamma_{k}}{p^k}$, and let $\isogroupof{\lambda}$ denote the (uniform) $\Gamma_{k}$-isotropy subgroup of its components. Then the representations of $\isogroupof{\lambda}$ afforded by the components of $\lambda$ are pairwise non-isomorphic irreducible representations of~$\isogroupof{\lambda}$.
\end{lemma}

\begin{corollary}   \label{corollary: one-sided inverse}
If $\lambda\in\Obj\weakfixedspecific{\Gamma_{k}}{p^k}$, then $FG(\lambda)=\lambda$.
\end{corollary}
\begin{proof}
By definition, $G(\lambda)=\isogroupof{\lambda}$, so the question is to find the
canonical isotypical decomposition of~$\isogroupof{\lambda}$.
Lemma~\ref{lemma: non-iso restrictions moved} says that all components of~$\lambda$
are non-isomorphic irreducible representations of~$\isogroupof{\lambda}$, so
in fact $F(\isogroupof{\lambda})=\lambda$.
\end{proof}

\begin{proof}[Proof of Lemma~\ref{lemma: non-iso restrictions moved}]
  Let $\rho$ denote the standard representation of $\Gamma_{k}$ on~$\complexes^{p^k}$.
  The action of $\Gamma_{k}/\isogroupof{\lambda}$ on $\class(\lambda)$ is free and transitive (the latter because $\Gamma_{k}$ acts irreducibly), so if we choose $v\in\class(\lambda)$, then $\rho$ is induced from the representation of $\isogroupof{\lambda}$ given by~$v$. We conclude that $v$ is an irreducible representation of~$\isogroupof{\lambda}$, since it induces the irreducible representation~$\rho$. The same is true for every other component of~$\lambda$, so the components of $\lambda$ are a decomposition of $\Cpk$ into $\isogroupof{\lambda}$-irreducibles.

  We can apply Frobenius reciprocity (see, for example,
  \cite[Theorem 9.9]{Knapp}) to conclude that:
\[
\left[\Ind_{\isogroupof{\lambda}}^{\Gamma_k}(v):\rho\right]
     =[\rho\vert_{\strut\isogroupof{\lambda}}:v].
\]
Because $\Ind_{\isogroupof{\lambda}}^{\Gamma_k}(v)\cong\rho$,
we conclude that $[\rho\vert_{\strut\isogroupof{\lambda}}:v ]=1$. However,
$\rho\vert_{\strut\isogroupof{\lambda}}$ is a direct sum of the irreducible $\isogroupof{\lambda}$-modules given by the components of~$\lambda$.
If any other component of $\lambda$ were isomorphic to
$v$ as a representation of~$\isogroupof{\lambda}$, then we would have
$[\rho\vert_{\strut\isogroupof{\lambda}}:v ]\geq 2$, contrary to the calculation above.
\end{proof}

In addition to showing that $F$ is a left inverse for~$G$,
Lemma~\ref{lemma: non-iso restrictions moved} also allows us to check that subgroups in the image of $G$ are actually proper coisotropic subgroups of~$\Gamma_{k}$.

\begin{lemma}   \label{lemma: isogroup is coisotropic}
If $\lambda$ is an object of~$\weakfixedspecific{\Gamma_{k}}{p^k}$,
then $\isogroupof{\lambda}$ is a proper coisotropic subgroup of~$\Gamma_{k}$.
\end{lemma}

\begin{proof}
We know that $\isogroupof{\lambda}$ is strictly contained in~$\Gamma_{k}$, because otherwise irreducibility of the action of $\Gamma_{k}$ would imply that $\lambda$ had only one component.

We have the following ladder of short exact sequences:
\[
\begin{CD}
1@>>> S^{1}@>>> \isogroupof{\lambda}@>>> W@>>> 1\\
@. @V{=}VV @VVV @VVV\\
1@>>> S^{1}@>>> \Gamma_{k}@>>> (\integers/p)^{2k}@>>> 1.
\end{CD}
\]
We must show that if $z\in W^{\perp}\subseteq(\integers/p)^{2k}$, then in fact $z\in W$. Recall that the symplectic form on $(\integers/p)^{2k}$ is given by the commutator pairing: if we denote lifts of $z$ and $w$ by $\zwiggle$ and $\wwiggle$, then the symplectic form evaluated on the pair $(z,w)$ is given by the commutator $[\zwiggle, \wwiggle]\in S^{1}$. Hence if $z$ pairs to $0$ with all elements of~$W$, it means that $\zwiggle$ is actually in the centralizer of $\isogroupof{\lambda}$ in~$\Gamma_{k}$. Thus is it sufficient for us to show that if $\zwiggle\in\Gamma_{k}$ centralizes~$\isogroupof{\lambda}$, then $\zwiggle\in\isogroupof{\lambda}$.

However, if $\zwiggle$ centralizes~$\isogroupof{\lambda}$ and $v\in\class(\lambda)$, then $\zwiggle$ gives a nontrivial $\isogroupof{\lambda}$-equivariant map between the $\isogroupof{\lambda}$-representations $v$ and $\zwiggle v$.
By Lemma~\ref{lemma: non-iso restrictions moved}, if $v\neq \zwiggle v$, then $v$ and $\zwiggle v$ are non-isomorphic irreducible representations of~$\isogroupof{\lambda}$, so Schur's Lemma tells us that there is no nontrivial $\isogroupof{\lambda}$-equivariant map. We conclude that $\zwiggle v=v$, so $\zwiggle\in\isogroupof{\lambda}$, as required.
\end{proof}

Finally, the last step is to show that the functors $F$ and $G$ are inverses of each other.

\begin{proof}[Proof of Theorem~\ref{theorem: fixed points Gamma_k}]
\mbox{}\hfill

The functors $F: H\mapsto\partitionof{H}$ and $G: \lambda\mapsto\isogroupof{\lambda}$
induce the desired homeomorphism, once we show that they are inverses of each other.
Corollary~\ref{corollary: one-sided inverse} already tells us that $FG(\lambda)=\lambda$. To finish the proof of the theorem, we must show
if $H$ is proper and coisotropic, then $GF(H)=H$, that is, the $\Gamma_{k}$-isotropy subgroup of $\partitionof{H}$ is $H$ itself.

By definition of~$\partitionof{H}$, the components of $\partitionof{H}$ are $H$-representations, so certainly $H\subseteq\isogroupof{\partitionof{H}}$.
Both $H$ and $\isogroupof{\partitionof{H}}$ are proper and coisotropic, by assumption and by Lemma~\ref{lemma: isogroup is coisotropic}, respectively. However,
a coisotropic subgroup of~$\Gamma_{k}$ is determined up to isomorphism by the dimension of its irreducible summands in the standard representation of~$\Gamma_{k}$
(Lemma~\ref{lemma: coisotropic determined by size}). Further, the components of $\partitionof{H}$ are irreducible representations for both $H$ (Lemma~\ref{lemma: coisotropic repns not iso}) and $\isogroupof{\partitionof{H}}$ (Lemma~\ref{lemma: non-iso restrictions moved}). Hence the irreducible summands of $H$ and $\isogroupof{\partitionof{H}}$ are actually the same,
and $H$ and $\isogroupof{\partitionof{H}}$ are isomorphic, and therefore equal.
\end{proof}

\section{Fixed points of $\Delta_{k}$ acting on $\Lpk$}
\label{section: Delta_k fixed points}

Let $\TitsOrdinary{k}$ denote the Tits building for $\GLof{k}$, that is, the poset of proper nontrivial subgroups of~$\Delta_{k}$. In this section, we prove the following result.

\begin{DeltaTheorem}
\DeltaFixedPointsTheoremText
\end{DeltaTheorem}

To set up the proof, we follow a similar strategy to~\cite[Section~9]{Banff2}.
Recall $\Uniform{\Delta_{k}}$ denotes the subposet of $\weakfixedspecific{\Delta_{k}}{p^k}$ consisting of objects with uniform $\Delta_{k}$-isotropy, and that
$\Uniform{\Delta_{k}}\hookrightarrow \weakfixedspecific{\Delta_{k}}{p^k}$
is a homotopy equivalence (Lemma~\ref{lemma: inclusion of uniform}).
We analyze $\Uniform{\Delta_{k}}$ in terms of two subposets.
\begin{definition}
\mbox{}\hfill
\begin{enumerate}
\item Let $\Nontransitive{\Delta_{k}}\subseteq\Uniform{\Delta_{k}}$ consist of objects $\lambda$ such that $\Delta_{k}$ does not act transitively on~$\class(\lambda)$.
\item Let $\Moving{\Delta_{k}}\subseteq\Uniform{\Delta_{k}}$ consist of objects $\lambda$ such that $\Delta_{k}$ acts nontrivially on~$\class(\lambda)$.
\end{enumerate}
\end{definition}

\begin{example}    \label{example: subcats}
Choose an orthonormal basis $E$ of $\Cpk$ on which $\Delta_{k}$ acts freely and transitively. (Recall that $\Delta_{k}$ is acting on $\Cpk$ by the regular representation.)  Let $\epsilon$ be the corresponding decomposition of $\Cpk$ into the lines, each line generated by an element of~$E$. Then $\epsilon$ is an object of  $\Moving{\Delta_{k}}$ but not of $\tallstrut\Nontransitive{\Delta_{k}}$,  and the same is true for $\glom{\epsilon}{K}$ for any proper subgroup $K\subseteq\Delta_{k}$.

Conversely, let $H$ be any nontrivial subgroup of~$\Delta_{k}$. Then $\partitionof{H}$ is an element of $\Nontransitive{\Delta_{k}}$
but not of $\Moving{\Delta_{k}}$.
\end{example}
\medskip

We observe that refinements of objects in $\Nontransitive{\Delta_{k}}$ are still in
$\Nontransitive{\Delta_{k}}$, and refinements of objects in $\Moving{\Delta_{k}}$ are still in $\Moving{\Delta_{k}}$. Further, every object of $\Uniform{\Delta_{k}}$ is in one of these two subposets. Hence we have a pushout diagram of nerves
\begin{equation}   \label{diag: union for uniform}
\begin{CD}
\NontransitiveDisp{\Delta_{k}}\cap  \MovingDisp{\Delta_{k}}
   @>>> \NontransitiveDisp{\Delta_{k}}\\
   @VVV @VVV\\
\MovingDisp{\Delta_{k}}@>>> \Uniform{\Delta_{k}}.
\end{CD}
\end{equation}
We assert that this diagram is in fact a homotopy pushout: that
the top row is a Reedy cofibration, and the bottom left space is
Reedy cofibrant. This is established by precisely the same argument as Proposition~9.11 of~\cite{Banff2}, with
the identity component of the centralizer of
$\Delta_{k}$ in~$\Upk$ in place of the centralizers that are applicable in that work. Essentially, the point is that in each simplicial dimension, one is looking at an inclusion of a subset of path components.

To prove Theorem~\ref{theorem: Tits building retract}, we will use the expected steps to show that
the nerve of $\Uniform{\Delta_{k}}$ has $\TitsOrdinary{k}^{\diamond}$ as a retract: finding a retraction map, exhibiting a corresponding inclusion, and showing that the inclusion and retraction compose to a self-equivalence of $\TitsOrdinary{k}^{\diamond}$.
Our first step is to use diagram \eqref{diag: union for uniform} to produce a map from the nerve of $\Uniform{\Delta_{k}}$ to the double cone on $\TitsOrdinary{k}$. Unlike the rest of the arguments in this paper, the map will not be realized on the categorical level, but only once we have passed to spaces by taking nerves. However, we begin on the categorical level. Define a function on object spaces,
\[
G: \NontransitiveDisp{\Delta_{k}}\cap  \MovingDisp{\Delta_{k}}\longrightarrow\TitsOrdinary{k}
\]
by the formula $G(\lambda)=\isogroupof{\lambda}$.

\begin{lemma}
The function $G$ defines a continuous functor.
\end{lemma}

\begin{proof}
First we need to check that $G(\lambda)$ is a proper, nontrivial subgroup of~$\Delta_{k}$. If $\lambda$ is an object of $\Moving{\Delta_{k}}$, then $\isogroupof{\lambda}$ is a proper subgroup of~$\Delta_{k}$. If $\isogroupof{\lambda}$ were trivial, then $\Delta_{k}$ would act freely
on~$\class(\lambda)$, implying that $\lambda$ is a decomposition of $\Cpk$ into $p^k$ lines, freely permuted by~$\Delta_{k}$. But then the action of $\Delta_{k}$ on $\class(\lambda)$ would be transitive, in contradiction of the assumption that $\lambda\in\Nontransitive{\Delta_{k}}$. Hence $G(\lambda)$ is a proper and nontrivial subgroup of~$\Delta_{k}$.
To check that $G$ defines a functor, we observe that if $\lambda\rightarrow\mu$ is a coarsening morphism in $\Uniform{\Delta_{k}}$, then $\isogroupof{\lambda}\subseteq\isogroupof{\mu}$.

The functor $G$ is defined on a subcategory of~$\Uniform{\Delta_{k}}$,
and its target category is discrete. Continuity of $G$ follows once we check that the assignment $\lambda\mapsto\isogroupof{\lambda}$ is constant on each path component
of~$\Uniform{\Delta_{k}}$. However, path components
of~$\Uniform{\Delta_{k}}\subseteq\weakfixedspecific{\Delta_{k}}{p^k}$ are orbits of the centralizer of~$\Delta_{k}$. If $c$ centralizes~$\Delta_{k}$, then $\isogroupof{c\lambda}=\isogroupof{\lambda}$.
Hence the assignment $\lambda\mapsto\isogroupof{\lambda}$ is constant on path components
of~$\Uniform{\Delta_{k}}$, and $G$ is therefore continuous.
\end{proof}

\begin{definition}   \label{definition: retraction map}
The map from the nerve of $\Uniform{\Delta_{k}}$ to $\TitsOrdinary{k}^{\diamond}$ is defined as the map of homotopy colimits arising from the following map of diagrams induced by~$G$ in the upper left corner:
\[
\begin{CD}
\left(\!\!\!\!
\begin{array}{ccc}
\Nontransitive{\Delta_{k}}\cap  \Moving{\Delta_{k}}
      &\longrightarrow&\Nontransitive{\Delta_{k}}\\
\downarrow\\
\Moving{\Delta_{k}}&&\phantom{\Uniform{\Delta_{k}}}
\end{array}\!\!\!\!
\right)
\\
\\
@VVV\\
\\
\left(\!\!
\begin{array}{ccc}
\TitsOrdinary{k} &\longrightarrow& *\\
\downarrow\\ 
*&&\phantom{\TitsOrdinary{k}^{\diamond}}
\end{array}\!\!
\right)
\end{CD}
\]
\end{definition}
\bigskip

The next piece of the puzzle is to define a map from $\TitsOrdinary{k}^{\diamond}$ into $\Uniform{\Delta_{k}}$. This map will be defined on the categorical level, that is, by taking the nerve of a functor between two categories, but we need a different categorical model for $\TitsOrdinary{k}^{\diamond}$ in order to define the map.
For this purpose, we recall some background on the edge subdivision of a category (also called a twisted arrow category). Suppose that $\Ccal$ is a category; define the ``edge subdivision" category $\twisted{\Ccal}$ of $\Ccal$  as follows:
\begin{enumerate}
\item Objects of $\twisted{\Ccal}$ are morphisms $X\rightarrow Y$ of $\Ccal$.
\item A morphism from $X\rightarrow Y$ to $C\rightarrow D$ is given by a \defining{twisted arrow}, that is, a commuting diagram
\[
\begin{CD}
X @>>> Y\\
@AAA @VVV\\
C @>>> D
\end{CD}
\]
\end{enumerate}

Note that if $\Ccal$ is a poset, then $\twisted{\Ccal}$ is a poset as well.

\begin{lemma}  \cite[Appendix 1]{Segal-Config-Spaces}
    \label{lemma: Sd equivalence}
The geometric realizations of $\twisted{\Ccal}$ and $\Ccal$
are naturally homeomorphic.
\end{lemma}

Recall that $\TitsOrdinary{k}$ is the poset of proper, non-trivial subgroups of $\Delta_k$. In what follows, let $\overline{\TitsOrdinary{k}}$ be the poset of {\it all} subgroups of $\Delta_k$. Note that $\twisted{\overline{\TitsOrdinary{k}}}$ has a final object $\{e\}\to \Delta_k$, but no initial object.

\begin{definition}
Let $\Tcal$ be the category $\twisted{\TitsOrdinary{k}}$ and let $\Tcal^{\diamond}$ be the category $\twisted{\overline{\TitsOrdinary{k}}}$ without the final object $\{e\}\to \Delta_k$. We will denote a generic object of $\twisted{\overline{\TitsOrdinary{k}}}$ by $H\subseteq K$.
\end{definition}

To justify the notation $\Tcal^{\diamond}$, we prove that the category $\Tcal^{\diamond}$ does in fact give a model for the unreduced suspension of the Tits building.
\begin{lemma}
The nerve of $\Tcal^{\diamond}$ is homeomorphic to
$\left|\TitsOrdinary{k}\right|^{\diamond}$.
\end{lemma}

\begin{proof}
We define $\NorthCone$ as the subposet of $\Tcal^{\diamond}$ consisting of pairs
$H\subseteq K$ where $H\neq\{e\}$. Likewise, we define
$\SouthCone$ as the subposet of $\Tcal^{\diamond}$ consisting of pairs $H\subseteq K$ where $K\neq\Delta_{k}$.

A straightforward check shows that if $H\subseteq K$
is an object of~$\NorthCone$ (respectively, $\SouthCone$), then $H\subseteq K$ can only be the target of morphisms from other objects in $\NorthCone$ (respectively, $\SouthCone$).
We conclude that a sequence of composable morphism that ends in $\NorthCone$ consists entirely of morphisms in $\NorthCone$, and similarly for $\SouthCone$.
Therefore on the level of nerves, we have
\[
\NorthCone\cup\SouthCone=\Tcal^{\diamond}
\]
Since the intersection $\NorthCone\cap\SouthCone$ is exactly~$\Tcal$, we have a pushout diagram of nerves
\begin{equation}  \label{diagram: Tits union}
\begin{CD}
\Tcal@>>>\NorthCone\\
@VVV @VVV \\
\SouthCone @>>> \Tcal^{\diamond}
\end{CD}
.
\end{equation}
Observe that $\NorthCone$ is the edge subdivision of $\TitsOrdinary{k}\cup\{\Delta_{k}\}$ (adding in the final
object $\{\Delta_{k}\}$ to the category being subdivided) and similarly for $\SouthCone$ (but by adding in the initial object~$\{e\}$). Hence
the nerves of $\NorthCone$ and $\SouthCone$ are each homeomorphic to a cone on the nerve of~$\Tcal$, and the result follows.
\end{proof}

%

We will define a functor
\[
\includeTits:\Tcal^{\diamond}\longrightarrow\Uniform{\Delta_{k}}.
\]
As in Example~\ref{example: subcats}, we fix an orthonormal basis of $\Cpk$ that is freely permuted by~$\Delta_{k}$, and let $\epsilon$ be the corresponding decomposition of $\Cpk$ into lines. For an object
$H\subseteq K$ of $\TitsOrdinary{k}^{\diamond}$, define $\includeTits$ by
\[
\includeTits(H\subseteq K) = \isorefine{\left(\glom{\epsilon}{K}\right)}{H}.
\]
Observe that this makes sense, because $H$ acts trivially on the set of components
of~$\glom{\epsilon}{K}$, so each component is a representation of $H$ and can itself be refined into $H$-isotypical components.

A couple of routine checks are required.

\begin{lemma}   \label{lemma: image uniform}
The image $\includeTits(H\subseteq K)$ is an object of $\Uniform{\Delta_{k}}$.
\end{lemma}

\begin{proof}
Since $\epsilon$ is stabilized by $\Delta_{k}$ and since $H$ and $K$ are normal in~$\Delta_{k}$, the operations of taking $K$-orbits and $H$-isotypical decomposition are stabilized by~$\Delta_{k}$. We also need to check that
$\includeTits(H\subseteq K)$ is a proper decomposition. If $K$ is a proper subgroup of $\Delta_{k}$, then $\glom{\epsilon}{K}$ is proper, so certainly any refinement of it is proper. If $K=\Delta_{k}$, then $\glom{\epsilon}{K}$ has just one component, all of~$\Cpk$, but since $H$ acts by copies of the regular representation, it acts non-isotypically. Hence $\includeTits(H\subseteq K)$ is a proper decomposition of~$\Cpk$.

To check whether $\includeTits(H\subseteq K)$ has uniform isotropy, first notice that since $K$ centralizes~$H$, an action of $K$ on a subspace $v$ fixes each of the canonical $H$-isotypical summands of $v$. Therefore $K$ stabilizes each component
of~$\isorefine{\left(\glom{\epsilon}{K}\right)}{H}$.
But the action of $\Delta_{k}/K$ on $\glom{\epsilon}{K}$ is free, so the action of $\Delta_{k}/K$ on $\isorefine{\left(\glom{\epsilon}{K}\right)}{H}$ is also free. Therefore $\isorefine{\left(\glom{\epsilon}{K}\right)}{H}$ has $K$ as the
$\Delta_{k}$-isotropy group of every component.
\end{proof}

\begin{lemma}
$\includeTits$ is a functor.
\end{lemma}

\begin{proof}
A morphism
$\left(H_{1}\subseteq K_{1}\right)\rightarrow\left(H_{2}\subseteq K_{2}\right)$
of $\Tcal^{\diamond}$ is given by a sequence of containments
$H_{2}\subseteq H_{1}\subseteq K_{1}\subseteq K_{2}$. We need to show that such a morphism gives rise to a coarsening morphism
\[
\isorefine{\left(\glom{\epsilon}{K_{1}}\right)}{H_{1}}
     \rightarrow
\isorefine{\left(\glom{\epsilon}{K_{2}}\right)}{H_{2}}.
\]
Certainly there is a coarsening morphism
$\glom{\epsilon}{K_{1}}\xrightarrow{\ c\ }\glom{\epsilon}{K_{2}}$, because
$K_{1}\subseteq K_{2}$. Components of both the source and the target of $c$ are
representations of~$H_{1}$, since $H_{1}\subseteq K_{1}\subseteq K_{2}$, so we can take the isotypical refinement of $c$ with respect to $H_{1}$ to obtain a morphism
\begin{equation}    \label{eq: coarsen isorefinements}
\isorefine{\left(\glom{\epsilon}{K_{1}}\right)}{H_{1}}
     \rightarrow
\isorefine{\left(\glom{\epsilon}{K_{2}}\right)}{H_{1}}.
\end{equation}
Following \eqref{eq: coarsen isorefinements} with the morphism
$\isorefine{\left(\glom{\epsilon}{K_{2}}\right)}{H_{1}}\rightarrow
      \isorefine{\left(\glom{\epsilon}{K_{2}}\right)}{H_{2}}$
 gives the desired result.

\end{proof}

Finally, we prove Theorem~\ref{theorem: Tits building retract} by considering the compositions of the maps of diagrams induced by $F$ and $G$.

\begin{proof}[Proof of Theorem~\ref{theorem: Tits building retract}]
The three diagrams we need to consider are
\mbox{}\smallskip
\begin{equation}    \label{diagram: subdivision diagram}
\left(\begin{array}{ccc}
\Tcal&\rightarrow&\NorthCone\\
\downarrow\\
\SouthCone 
\end{array}\right)
\end{equation}
%
\\
mapping on all three corners via
$F: (H\subseteq K)\mapsto\isorefine{\left(\glom{\epsilon}{K}\right)}{H}$ to
\\
\begin{equation}    \label{diagram: decompositions union}
\left(\!\!\!\!
\begin{array}{ccc}
\Nontransitive{\Delta_{k}}\cap  \Moving{\Delta_{k}}
      &\longrightarrow&\Nontransitive{\Delta_{k}}\\
\downarrow\\
\Moving{\Delta_{k}}&&\phantom{\Uniform{\Delta_{k}}}
\end{array}\!\!\!\!
\right)
\end{equation}
\\
which then has a map of nerves induced by
$\strut G:\lambda\mapsto\isogroupof{\lambda}$ to
\begin{equation}    \label{diagram: topological pushout}
\left(\!\!
\begin{array}{ccc}
\TitsOrdinary{k} &\longrightarrow& *\\
\downarrow\\ 
*&&\phantom{\TitsOrdinary{k}^{\diamond}}
\end{array}\!\!
\right).
\end{equation}

We first need to check that the corners of diagram~\eqref{diagram: subdivision diagram} map to the corners of diagram~\eqref{diagram: decompositions union} as claimed.
For the lower left-hand corner, notice that if $H\subseteq K\neq\Delta_{k}$ is an object of $\SouthCone$, then there is a coarsening morphism
\[
\isorefine{\left(\glom{\epsilon}{K}\right)}{H}
\longrightarrow
\glom{\epsilon}{K}.
\]
Since the set of components of
$\glom{\epsilon}{K}$
has more than one element and a transitive (hence necessarily nontrivial) action of~$\Delta_{k}$,
the action of $\Delta_{k}$ on the components of $\isorefine{\left(\glom{\epsilon}{K}\right)}{H}$ is also nontrivial.

For the upper right-hand corner of diagram~\eqref{diagram: decompositions union},
if $\{e\}\neq H\subseteq K$ is an object of $\NorthCone$, then we have a
coarsening morphism
\[
\isorefine{\left(\glom{\epsilon}{K}\right)}{H}
\longrightarrow
\isorefine{\left(\glom{\epsilon}{\Delta_{k}}\right)}{H}=\partitionof{H}.
\]
However, $\partitionof{H}$ has more than one component because $H$ is nontrivial, and
$\Delta_{k}$ acts trivially (hence nontransitively) on $\class\left(\partitionof{H}\right)$ because $H$ is central in~$\Delta_{k}$. Hence the action of $\Delta_{k}$ on the components of~$\isorefine{\left(\glom{\epsilon}{K}\right)}{H}$
cannot be transitive either.

The maps given between diagrams~\eqref{diagram: subdivision diagram}, \eqref{diagram: decompositions union}, and~\eqref{diagram: topological pushout} give maps on homotopy pushouts:
\begin{equation}    \label{eq: inclusion retraction}
\Tcal^{\diamond}\longrightarrow\Uniform{\Delta_{k}}
                \longrightarrow\TitsOrdinary{k}^{\diamond}.
\end{equation}
To prove the theorem, it is sufficient to show that the composition of
diagrams~\eqref{diagram: subdivision diagram}, \eqref{diagram: decompositions union}, and~\eqref{diagram: topological pushout} gives a homotopy equivalence of nerves on the upper left-hand corner,
\[
\Tcal\longrightarrow \Nontransitive{\Delta_{k}}\cap  \Moving{\Delta_{k}}
     \longrightarrow \TitsOrdinary{k}.
\]
However, the composition takes an object $H\subseteq K$ of~$\Tcal$ to the isotropy subgroup of $\isorefine{\left(\glom{\epsilon}{K}\right)}{H}$, which is $K$ itself,
as in the proof of Lemma~\ref{lemma: image uniform}.
Hence the composition $\Tcal\rightarrow\TitsOrdinary{k}$ maps $(H\subseteq K)$ to~$K$, which induces an equivalence of nerves by
\cite[p. 94]{Quillen-Battelle}.
\end{proof}


\section{Conjectures}
\label{section: conjectures}

In the introduction, we presented
a general conjecture regarding the $\Uof{n-1}$-equivariant
homotopy type of~$\Lcal_{n}$. Recall that $\Pcal_{n}$ denotes
the poset of proper nontrivial partitions of a set of $n$ elements and
$\Pcal_{n}^{\diamond}$ denotes its unreduced suspension.
The group $\Sigma_{n}$ is embedded in $\Uof{n}$ via the standard (permutation) representation,
and $\RepresentationSphere{n}$ denotes the representation sphere of the reduced standard representation of $\Sigma_{n}$ on $\complexes^{n}$.

\begin{ConjectureBranching}
\ConjectureBranchingText
\end{ConjectureBranching}

In this section, we show that the following conjecture follows
from Conjecture~\ref{conjecture: branching}.

\begin{FixedPointConjecture}
\FixedPointConjectureText
\end{FixedPointConjecture}
The case $k=1$ is computed explicitly in Example~\ref{example: general linear}.

Dividing $C_{\Upkfoot}(\Delta_{k})\cong\left(\Uof{1}\right)^{p^k}$
by the subgroup $\Delta_{k}\times S^{1}$ still leaves us with a torus,
so we have
a homeomorphism $\Cwiggle\cong \left(S^{1}\right)^{p^k-1}$. Recall that
$\TitsOrdinary{k}^{\diamond}$ is a wedge of spheres of dimension $k-1$.
Conjecture~\ref{conjecture: fixed} would tell us that for $k>1$, the fixed point space
$\weakfixedspecific{\Delta_{k}}{p^{k}}$ is a wedge of spheres of varying dimensions. Further, by the join formula from
\cite{Banff2}, we have
\[
\weakfixedspecific{\Gamma_{s}\times\Delta_{t}}{p^{s+t}}
    \simeq\weakfixedspecific{\Gamma_{s}}{p^s}
    \ast\weakfixedspecific{\Delta_{t}}{p^{t}},
\]
which would also be a wedge of spheres (of varying dimensions for $t>0$) provided
that either $s>0$ or $t>1$.

Recall that we are considering $\Uof{p^k-1}\subset\Upk$ as the symmetries
of the orthogonal complement of the diagonal $\complexes\subset\Cpk$. The subgroup
$\Delta_{k}\subset\Sigma_{p^k}$ is a subgroup of $\Uof{p^k-1}$ with this embedding.
To show that Conjecture~\ref{conjecture: fixed} follows from Conjecture~\ref{conjecture: branching}, we
need to calculate the fixed points of
$\Delta_{k}\subset\Sigma_{p^{k}}$ acting on
\begin{equation} \label{eq: our fixed points}
\Uof{p^{k}-1}_{+}\wedge_{\Sigma_{p^k}}
\left(\Pcal_{p^k}^\diamond \wedge \RepresentationSphere{p^k}\right).
\end{equation}
In general, the fixed points of $D\subseteq G$ on a space with an action of $H\subseteq G$ induced up to $G$ is
\begin{equation} \label{eq: general fixed points}
\left(G\times_{H}X\right)^{D}=\bigcup_{[g]\in N(D;H)/H} \{g\}\times X^{g^{-1}Dg},
\end{equation}
where $N_{G}(D;H)=\{g\in G: g^{-1}Dg\subseteq H\}$. Thus we need
$N_{\Uoffoot{p^k-1}}(\Delta_{k};\Sigma_{p^{k}})$.

To calculate $N_{\Upkfoot}\left(\Delta_{k};\Sigma_{p^{k}}\right)$, suppose that $u\in\Upk$ satisfies
$u^{-1}\Delta_{k}u\subset\Sigma_{p^k}\subset\Uof{p^k}$, which means that all elements of
$u^{-1}\Delta_{k}u$ are permutation matrices.
The character of $u^{-1}\Delta_{k}u$ is the same as that of $\Delta_{k}$, i.e., zero on all nonidentity elements, which tells us that $u^{-1}\Delta_{k}u$ acts freely and hence transitively on $\{1,...,p^k\}$.
But then $\Delta_{k}$ and $u^{-1}\Delta_{k}u$ are both transitive elementary abelian $p$-subgroups of $\Sigma_{p^{k}}$, which means that they are conjugate inside of $\Sigma_{p^k}$ itself. So there exists $\sigma\in\Sigma_{p^{k}}$ such that
$\sigma^{-1}\Delta_{k}\sigma=u^{-1}\Delta_{k}u\subset\Sigma_{p^k}$.

However, all automorphisms of $\Delta_{k}$ are realized by the action of its normalizer in~$\Sigma_{p^k}$. By changing the choice of $\sigma$ if necessary, we can actually make the stronger assertion that $\sigma$ and $u$ induce the same automorphism of $\Delta_{k}$, i.e.
$\sigma^{-1}d\sigma=u^{-1}du$ for all $d\in\Delta_{k}$.
Thus $u\sigma^{-1}$ centralizes every $d\in\Delta_{k}$, and
$u$ is in the coset
$C_{\Uoffoot{p^{k}}}\left(\Delta_{k}\right)\,\sigma$.
We conclude that
\[
N_{\Upkfoot}\left(\Delta_{k};\Sigma_{p^{k}}\right)
    = \bigcup_{\sigma\in \Sigma_{p^k}}
       C_{\Uoffoot{p^{k}}}\left(\Delta_{k}\right)\,\sigma.
\]
Since the centralizer of $\Delta_k$ in $\Sigma_{p^k}$ is $\Delta_k$ itself,
$C_{\Uoffoot{p^{k}}}\left(\Delta_{k}\right)
     \cap \Sigma_{p^k}=\Delta_k$.
It follows that the formula for
$N_{\Upkfoot}\left(\Delta_{k};\Sigma_{p^{k}}\right)$ can be rewritten as
\[
N_{\Upkfoot}\left(\Delta_{k};\Sigma_{p^{k}}\right)
     = C_{\Uoffoot{p^{k}}}\left(\Delta_{k}\right)
          \times_{\Delta_k} \Sigma_{p^k}.
\]

Next we restrict to $\Uof{p^k-1}\subset\Upk$, and observe that
\[
N_{\Uoffoot{p^k-1}}\left(\Delta_{k};\Sigma_{p^{k}}\right)
    =N_{\Upkfoot}\left(\Delta_{k};\Sigma_{p^{k}}\right)
     \,\cap\,\Uof{p^k-1}.
\]
We have already found that
$N_{\Upkfoot}\left(\Delta_{k};\Sigma_{p^{k}}\right)$
is a union of cosets
$C_{\Uoffoot{p^{k}}}\left(\Delta_{k}\right)\,\sigma$,
and $\sigma\in\Sigma_{p^{k}}\subset\Uof{p^k-1}$,
so we need only compute the intersection
of $C_{\Uoffoot{p^{k}}}\left(\Delta_{k}\right)$ with $\Uof{p^k-1}$.
Recall that
$C_{\Uoffoot{p^{k}}}\left(\Delta_{k}\right)
    =\left(\Uof{1}\right)^{p^{k}}$,
where each copy of $\Uof{1}$ acts on a different irreducible representation of $\Delta_{k}$ on $\complexes^{p^{k}}$.
However, $\Uof{p^k-1}$ is the symmetry group of the orthogonal complement of the diagonal
$\complexes\subset\complexes^{p^k}$, and the diagonal is actually the trivial representation of $\Delta_{k}$, so we find
\[
C_{\Upkfoot}\left(\Delta_{k}\right)\cap\Uof{p^k-1}=\left(\Uof{1}\right)^{p^{k}-1},
\]
where each $\Uof{1}$ acts on a different nontrivial irreducible representation of~$\Delta_{k}$,
and
\[
N_{\Uoffoot{p^k-1}}\left(\Delta_{k};\Sigma_{p^{k}}\right)
     =\bigcup_{\sigma\in\Sigma_{p^k}}
           \left(\Uof{1}\right)^{p^{k}-1}\,\sigma
     =\left(\Uof{1}\right)^{p^{k}-1}\times_{\Delta_k}\Sigma_{p^k}.
\]
Taking the quotient by~$\Sigma_{k}$, we find that the indexing
set in \eqref{eq: general fixed points} applied to
\eqref{eq: our fixed points} is
\[
N_{\Uoffoot{p^k-1}}(\Delta_{k};\Sigma_{p^{k}})/\Sigma_{p^{k}} = \left(\Uof{1}\right)^{p^{k}-1}/{\Delta_k}.
\]

To finish the calculation, we note that $\left(\RepresentationSphere{p^k}\right)^{\Delta_k}\cong S^0$ and we recall that
by~\cite[Lemma 10.1]{ADL2}, $\left(\Pcal_{p^k}^\diamond\right)^{\Delta_k}$ is
equivalent to $\TitsOrdinary{k}^\diamond$.
Assembling all the pieces,
\begin{align*}
\left[\Uof{p^{k}-1}_{+}\wedge_{\Sigma_{p^k}}
\left(\Pcal_{p^k}^\diamond \wedge \RepresentationSphere{p^k}\right)\right]^{\Delta_{k}}\\
&\hspace{-10em} = \bigcup_{[g]\in N_{\Uoffoot{p^k-1}}\left(\Delta_{k};\Sigma_{p^{k}}\right)/\Sigma_{p^{k}}}
\{[g]\}_{+}\wedge\left(\Pcal_{p^k}^{\diamond}
       \wedge \RepresentationSphere{p^k}\right)^{\Delta_{k}}\\
&\hspace{-10em}\cong \left(\Uof{1}^{p^{k}-1}\right)/\Delta_{k}\mbox{}_{+}
    \wedge
    \left(\Pcal_{p^k}^{\diamond}\wedge \RepresentationSphere{p^k}\right)^{\Delta_{k}}\\
&\hspace{-10em}\cong C_{\Upkfoot}
\left(\Delta_{k}\right)/\left(\Delta_{k}\times S^{1}\right)_{+}
           \wedge \TitsOrdinary{k}^{\diamond},
\end{align*}
where the $S^1$ in the last line is the center of~$\Upk$.

We conclude that Conjecture~\ref{conjecture: fixed} follows from Conjecture~\ref{conjecture: branching}.

\begin{example}     \label{example: general linear}
We can compute $\weakfixedspecific{\Delta_{1}}{p}$ explicitly.
(In fact, this is done
via completely elementary manipulations in \cite{Banff1} for $p=2$.)
There are two types of decompositions $\lambda$ in~$\weakfixedspecific{\Delta_{1}}{p}$:

\emph{(i)} $\Delta_{1}$ acts freely on $\class(\lambda)$, in which case $\lambda$ has $p$ components, each of which is a line;

\emph{(ii)} $\Delta_{1}$ acts trivially on $\class(\lambda)$, in which case each component of $\lambda$ is a representation of~$\Delta_{1}$.

In the first situation, the decompositions of $\complexes^{p}$ into lines that are freely
(and therefore transitively) permuted by $\Delta_{1}$ have no refinements,
and also no coarsenings that are stabilized by $\Delta_{1}$.
We assert that they are all in a single orbit of
$C_{\Uoffoot{p}}(\Delta_{1})\cong\left(\Uof{1}\right)^{p}$.
For suppose that $\lambda$ and $\mu$ are such decompositions,
with $\class(\lambda)=\{v_{1},...,v_{p}\}$
and $\class(\mu)=\{w_{1},...,w_{p}\}$. Choose an isomorphism $f$ from $v_{1}$ and $w_{1}$, and consider
the unique extension of $f$ to a $\Delta_{1}$-equivariant map $u\in\Uof{p}$. Then $u\lambda=\mu$, and $u$ centralizes $\Delta_{1}$ by construction. Some linear algebra allows us to show that if
$u\in C_{\Uoffoot{p}}(\Delta_{1})\cong\left(\Uof{1}\right)^{p}$ stabilizes~$\lambda$, then
$u\in S^{1}\times\Delta_{1}$, so this component of the object space is homeomorphic to
$C_{\Uoffoot{p}}(\Delta_{1})/\left(S^{1}\times\Delta_{1}\right)$.

On the other hand, the decompositions of $\complexes^{p}$ whose components are each stabilized by $\Delta_{1}$ are sums of the $p$ distinct one-dimensional representations of $\Delta_{1}$ in its regular representation on~$\complexes^{p}$. There are coarsening morphisms between such decompositions, but there are no morphisms from such decompositions to those of the paragraph above. There is an initial object in the subcategory of objects $\lambda$ in $\weakfixedspecific{\Delta_{1}}{p}$ with trivial action on $\class(\lambda)$, namely the canonical decomposition of $\complexes^{p}$ into the lines that are the irreducible representations of~$\Delta_{1}$.

Hence we can actually deduce that
\begin{align*}
\weakfixedspecific{\Delta_{1}}{p}
     & \cong
    \Cone\left(\Pcal_{p}\right)
           \sqcup C_{\Uoffoot{p}}(\Delta_{1})/\left(S^{1}\times\Delta_{1}\right)\\
     & \simeq
     C_{\Uoffoot{p}}(\Delta_{1})/\left(S^{1}\times\Delta_{1}\right)_{+}
             \wedge\TitsOrdinary{1}^{\diamond}
\end{align*}
because $\TitsOrdinary{1}=\emptyset$. The result is in conformity with
Conjecture~\ref{conjecture: fixed}, and is
also in agreement with the calculation for $p=2$
in \cite{Banff1}, where it was found that
$\weakfixedspecific{\integers/2}{2}\cong \ast\sqcup S^{1}$.
\end{example}

\bibliographystyle{amsalpha}

\begin{thebibliography}{{Aro}15}

\bibitem[AB08]{Abramenko-Brown-Buildings}
Peter Abramenko and Kenneth~S. Brown, \emph{Buildings}, Graduate Texts in
  Mathematics, vol. 248, Springer, New York, 2008, Theory and applications.
  \MR{2439729}

\bibitem[AB18]{Arone-Brantner}
G.~{Arone} and L.~{Brantner}, \emph{{The Action of Young Subgroups on the
  Partition Complex}}, ArXiv e-prints (2018).

\bibitem[AD01]{Arone-Dwyer}
G.~Z. Arone and W.~G. Dwyer, \emph{Partition complexes, {T}its buildings and
  symmetric products}, Proc. London Math. Soc. (3) \textbf{82} (2001), no.~1,
  229--256. \MR{1794263 (2002d:55003)}

\bibitem[ADL16]{ADL2}
G.~Z. Arone, W.~G. Dwyer, and K.~Lesh, \emph{Bredon homology of partition
  complexes}, Doc. Math. \textbf{21} (2016), 1227--1268. \MR{3578208}

\bibitem[AL07]{Arone-Lesh-Crelle}
Gregory~Z. Arone and Kathryn Lesh, \emph{Filtered spectra arising from
  permutative categories}, J. Reine Angew. Math. \textbf{604} (2007), 73--136.
  \MR{2320314 (2008c:55013)}

\bibitem[AL10]{Arone-Lesh-Fundamenta}
\bysame, \emph{Augmented {$\Gamma$}-spaces, the stable rank filtration, and a
  {$bu$} analogue of the {W}hitehead conjecture}, Fund. Math. \textbf{207}
  (2010), no.~1, 29--70. \MR{2576278}

\bibitem[AM99]{Arone-Mahowald}
Greg Arone and Mark Mahowald, \emph{The {G}oodwillie tower of the identity
  functor and the unstable periodic homotopy of spheres}, Invent. Math.
  \textbf{135} (1999), no.~3, 743--788. \MR{1669268}

\bibitem[Aro02]{Arone-Topology}
Greg Arone, \emph{The {W}eiss derivatives of {$B{\rm O}(-)$} and {$B{\rm
  U}(-)$}}, Topology \textbf{41} (2002), no.~3, 451--481. \MR{1910037}

\bibitem[{Aro}15]{Arone-Branching}
G.~{Arone}, \emph{{A branching rule for partition complexes}}, ArXiv e-prints
  (2015).

\bibitem[BJL{\etalchar{+}}]{Banff2}
Julia~E. Bergner, Ruth Joachimi, Kathryn Lesh, Vesna Stojanoska, and Kirsten
  Wickelgren, \emph{Classification of problematic subgroups of $u(n)$},
  arXiv:1407.0062 [math.AT], to appear in Trans. Amer. Math. Soc.

\bibitem[BJL{\etalchar{+}}15]{Banff1}
\bysame, \emph{Fixed points of {$p$}-toral groups acting on partition
  complexes}, Women in topology: collaborations in homotopy theory, Contemp.
  Math., vol. 641, Amer. Math. Soc., Providence, RI, 2015, pp.~83--96.
  \MR{3380070}

\bibitem[HL18]{Hopkins-Lawson}
Michael~J. Hopkins and Tyler Lawson, \emph{Strictly commutative complex
  orientation theory}, Math. Z. \textbf{290} (2018), no.~1-2, 83--101.
  \MR{3848424}

\bibitem[Kna96]{Knapp}
Anthony~W. Knapp, \emph{Lie groups beyond an introduction}, Progress in
  Mathematics, vol. 140, Birkh\"auser Boston, Inc., Boston, MA, 1996.
  \MR{1399083 (98b:22002)}

\bibitem[May99]{May-Tel-Aviv}
J.~P. May, \emph{Equivariant orientations and {T}hom isomorphisms}, Tel {A}viv
  {T}opology {C}onference: {R}othenberg {F}estschrift (1998), Contemp. Math.,
  vol. 231, Amer. Math. Soc., Providence, RI, 1999, pp.~227--243. \MR{1707345}

\bibitem[Oli94]{Oliver-p-stubborn}
Bob Oliver, \emph{{$p$}-stubborn subgroups of classical compact {L}ie groups},
  J. Pure Appl. Algebra \textbf{92} (1994), no.~1, 55--78. \MR{1259669
  (94k:57055)}

\bibitem[Qui73]{Quillen-Battelle}
Daniel Quillen, \emph{Higher algebraic {$K$}-theory. {I}}, Algebraic
  {$K$}-theory, {I}: {H}igher {$K$}-theories ({P}roc. {C}onf., {B}attelle
  {M}emorial {I}nst., {S}eattle, {W}ash., 1972), Springer, Berlin, 1973,
  pp.~85--147. Lecture Notes in Math., Vol. 341. \MR{0338129}

\bibitem[Seg73]{Segal-Config-Spaces}
Graeme Segal, \emph{Configuration-spaces and iterated loop-spaces}, Invent.
  Math. \textbf{21} (1973), 213--221. \MR{0331377}

\end{thebibliography}
\newcommand{\etalchar}[1]{$^{#1}$}
\providecommand{\bysame}{\leavevmode\hbox to3em{\hrulefill}\thinspace}
\providecommand{\MR}{\relax\ifhmode\unskip\space\fi MR }
\providecommand{\MRhref}[2]{%
  \href{http://www.ams.org/mathscinet-getitem?mr=#1}{#2}
}
\providecommand{\href}[2]{#2}

\end{document}